\documentclass[10pt]{article}
\usepackage{leftidx}
\usepackage{amssymb}
\usepackage{color}
\usepackage{amsmath}
\usepackage{amsfonts}
\usepackage{graphicx}
\usepackage{amsthm}
\usepackage{float}
\usepackage{xspace}
\usepackage{stmaryrd}
\usepackage{multirow}
\usepackage{psfrag}

\usepackage{mathtools}
\mathtoolsset{showonlyrefs}

\newcommand{\RR}{\mathbb{R}}

\newcommand{\x}[1]{\operatorname{#1}}
\newcommand{\dom}{\x{D}}
\newcommand{\spec}{\x{Spec}}

\renewcommand{\frak}[1]{\mathfrak{#1}}

\newcommand{\dx}{\mathrm{d}x}

\renewcommand{\geq}{\geqslant}
\renewcommand{\leq}{\leqslant}

\newtheorem{thm}{Theorem}[section]
\newtheorem{lem}[thm]{Lemma}

\newtheorem{rem}[thm]{Remark}
\newtheorem{col}[thm]{Corollary}

\title{On the quality of complementary bounds for eigenvalues}

\author{Lyonell Boulton$^1$ \and Aatef Hobiny$^2$}

\date{4th August 2014}

\begin{document}

\maketitle

\footnotetext[1]{Corresponding author. Email \texttt{L.Boulton@hw.ac.uk} tel +441314513973 fax +444513432.}
\footnotetext[12]{Department of Mathematics and 
Maxwell Institute for Mathematical Sciences, Heriot-Watt University, Edinburgh, EH14 4AS, UK.}
\footnotetext[2]{Department of Mathematics, Faculty of Science, King Abdulaziz University, P.O. Box. 80203, Jeddah 21589, Saudi Arabia.}

\abstract{A concrete formulation of the Lehmann-Maehly-Goerisch method for semi-definite self-adjoint operators with compact resolvent is considered. Precise rates of convergence are determined in terms of how well the trial spaces capture the spectral subspace of the operator. Optimality of the
choice of a shift parameter which is intrinsic to the method is also examined. The main theoretical findings are illustrated by means of a few numerical experiments involving one-dimensional Schr{\"o}dinger operators.}

\

\

\


\pagebreak

\section{Introduction}

The neat formulation by Zimmermann and Mertins \cite{1995Zimmer} (see also \cite[Section~6]{2004Davies}) of the Lehmann-Maehly-Goerisch method \cite{1985Goerisch,1980Goerisch,1949Lehmann,1950Lehmann,1952Maehly} (see also \cite[Chapter~4.11]{1974Weinberger}), has recently shown to be a reliable tool for computing eigenvalue enclosures \cite{2001Behnke,2004Davies,2009Behnke,2011Boulton,2013Boulton,2014Boulton}. In its most basic framework, this formulation involves fixing a ``shift'' parameter $t\in \mathbb{R}$ and then characterising the spectrum which is adjacent to $t$ by means of
a combination of the Variational Principle with the Spectral Mapping Theorem.

The present paper is devoted to re-examining this most basic setting for semi-definite self-adjoint operators with a compact resolvent. Two main contributions are to be highlighted. On the one hand, we determine how the choice of $t$ affects the quality of the eigenvalue bound. On the other hand, we establish explicit  convergence estimates, in terms of how well spectral subspaces in a neighbourhood of $t$ are captured by the underlying trial subspaces. The latter is closely linked with 
a similar convergence analysis pursued in \cite{20133Boulton} for the so-called quadratic method. 

In addition to these theoretical contributions, we apply our findings in the detailed study of a concrete model. For this purpose we consider computation of upper and lower bounds for the eigenvalues of one-dimensional Schr{\"o}dinger operators with potential singular at infinity by means of the finite element method. Similar ideas have been realised for the Helmholtz equation \cite{2001Behnke}, calculation of sloshing frequencies \cite{2009Behnke} and the MHD operator \cite{2011Boulton}. See also the recent  manuscripts \cite{2013Boulton} and \cite{2014Boulton}.

The first part of the paper is concerned with the abstract theory. 
Section \ref{zimesec} is devoted to a formulation of the most basic framework in the approach described in \cite{1995Zimmer}. In Section~\ref{impzimesec} we determine how the choice of the parameter $t$ affects the quality of the eigenvalue bounds
(Theorem~\ref{lem52}). Properties of convergence are then established in Section~\ref{convergencezimmermann}. Our main contribution in this respect is summarised by Theorem~\ref{pro63}.

The second part of the paper is devoted to the concrete model. 
Section~\ref{zm_numerical} describes how to discretise the Schr{\"o}dinger operator by means of the finite element method. There we establish precise convergence rates in this particular case. In Section~\ref{numex} we include various computational experiments performed on the harmonic and the anharmonic oscillators.

\subsection*{Notation}

Everywhere below $\mathcal{H}$ denotes a Hilbert space with inner product $\langle \cdot, \cdot \rangle$ and norm $\left\|\cdot\right\|$. The  self-adjoint operator $A=A^*:\dom (A) \longrightarrow \mathcal{H}$, will always be assumed to be semi-bounded below and its resolvent operator \[(A-z)^{-1}:\mathcal{H} \longrightarrow \dom (A)\] will be assumed to be compact for one and hence all $z\in \mathbb{C}$ outside the spectrum. Under these hypotheses, 
the spectrum of $A$, $ \spec(A)$, is always an increasing sequence of isolated eigenvalues of finite multiplicity accumulating at $+\infty$. We will write
\[
\spec(A)=\{\lambda_1\leq \lambda_2\leq \ldots \}
\]
counting multiplicities with the index.

Here, and everywhere else in this paper, $t$ will denote a real parameter. We will often leave implicit the dependence of some of the $t$-dependant quantities, whenever this is sufficiently clear from the context. This will be so the case especially in the proofs of some of the main statements.

For $u,\,v\in \dom(A)$, we will write
\[
     \frak{a}^0(u,v)=\langle u,v \rangle \qquad \frak{a}^1(u,v)=\langle A u,v \rangle \qquad
      \frak{a}^2(u,v)=\langle Au,Av \rangle
\]
\[
 \frak{a}^1_t(u,v)=\langle (A-t) u,v \rangle \qquad  \frak{a}^2_t(u,v)=\langle (A-t)u,(A-t)v \rangle
\]
so that $\frak{a}^1=\frak{a}^1_0$ and $\frak{a}^2=\frak{a}^2_0$. In general, \[\frak{a}^2_t:\dom (A)\times \dom (A)\longrightarrow \mathbb{C},\] are closed quadratic forms, but this is not necessarily the case for $\frak{a}^1_t$ on $\dom (A)\times \dom (A)$. Moreover $(\dom (A),\frak{a}^2_t)$ defines a Hilbert space if and only if $t\not\in \spec(A)$. We will occasionally write $|u|_t=\frak{a}_t^2(u,u)^{1/2}$.  

Let $\mathcal{L} \subset \dom(A)$ be a given subspace of dimension $n$ such that 
\[
\mathcal{L}=\x{Span} \{b_j\}^n_{j=1}
\]
for a suitable linearly independent set $\{b_j\}^n_{j=1}$.
We will write 
\[
\mathbf{A}_{l}=\left[ \frak{a}^l (b_j,b_k)\right]^n_{jk=1} \in \mathbb{C}^{n \times n}
\quad \text{and} \quad \mathbf{A}_{l,t}=\left[ \frak{a}^l_t(b_j,b_k)\right]^n_{jk=1} \in \mathbb{C}^{n \times n}.
\]

\section{Complementary eigenvalue bounds}   \label{zimesec}

The most basic setting of the strategy established in the paper \cite{1995Zimmer} can be summarised as follows. Let  $t\in \RR$ be fixed.  In order to give certified bounds  for the spectrum of $A$ in the vicinity of $t$, we seek for the eigenvalues of the following problem: 
find $\tau\in \RR$ and $u\in \mathcal{L} \setminus \{0\}$, such that
\begin{equation} \label{eq31}
 \tau \;  \frak{a}^2_t(u,v)=  \frak{a}^1_t(u,v) \quad \forall v\in \mathcal{L}.
\end{equation}

\subsection{Adjacent eigenvalues}

Consider the case of the eigenvalues which are immediately adjacent to $t$, both to the left and to the right. Let $x < y$ be such that \[(x,y)\cap \spec(A)=\{\lambda\}.\] Assume that $\mathcal{L}$ is such that 
\begin{equation} \label{eq32}
 x < \max_{u\in\mathcal{L}\setminus \{0\}} \frac{\frak{a}^1(u,u)}{\frak{a}^0(u,u)} \quad \text{and} \quad y > \min_{u\in\mathcal{L}\setminus \{0\}} \frac{\frak{a}^1(u,u)}{\frak{a}^0(u,u)}. 
\end{equation}
Then \eqref{eq31} for $t=x$ and $t=y$ leads to upper and lower bounds for $\lambda$. Indeed, denote the extremal eigenvalues of \eqref{eq31} by
\begin{equation}  \label{eq33}
  \tau^-_1(y)=\min_{u\in\mathcal{L}}\frac{\frak{a}^1_y(u,u)}{\frak{a}^2_y(u,u)}
\quad \text{and} \quad 
\tau^+_1(x)=\max_{u\in\mathcal{L}} \frac{\frak{a}^1_x(u,u)}{\frak{a}^2_x(u,u)}.
\end{equation}
The following statement is a simplified version of \cite[Theorem~2.4]{1995Zimmer}. See also \cite[Theorem~11]{2004Davies}
and Lemma~\ref{lem42} below.

\begin{lem} \label{lem31}
If \eqref{eq32} holds true, then $ \tau^+_1(x)>0$, $\tau^-_1(y)<0$ and
\begin{equation*}
 y+\frac{1}{\tau^-_1(y)} \leqslant \lambda \leqslant x+\frac{1}{\tau^+_1(x)}.
\end{equation*}
\end{lem}
\begin{proof} Condition \eqref{eq32} on $y$ implies immediately $\tau^-_1(y)<0$. Let \[\widehat{\mathcal{L}}=(A-y)\mathcal{L}.\] According to the Spectral Mapping Theorem and the Min-max Principle, 
\begin{align*}
 (\lambda-y)^{-1}&=\min[\spec(A-y)^{-1}]=\min_{v\in \mathcal{H}} \frac{\langle (A-y)^{-1}v,v\rangle }{\|v\|^2} \\
& \leqslant \min_{v\in \widehat{\mathcal{L}}} \frac{\langle (A-y)^{-1}v,v\rangle }{\|v\|^2}=\min_{u\in\mathcal{L}}\frac{\langle (A-y)u,u\rangle }{\langle (A-y)u,(A-y)u\rangle }\\
&=\min_{u\in\mathcal{L}}\frac{\frak{a}^1_y(u,u)}{\frak{a}^2_y(u,u)}=\tau^-_1(y).
\end{align*}
This ensures the inequality on the left hand side. 

The complementary inequality and the statement for $x$ are shown analogously.
\end{proof}

\subsection{Further eigenvalues}
Let $\ell\equiv \ell(t)$  be the number of eigenvalues of $A$ which are below $t$ counting multiplicity. Here and elsewhere we convey in writing  $\ell(t)=0$ when $t<\lambda_1$, as well as $\lambda_0=-\infty$.

As we shall see next, according to the Min-max Principle, the eigenvalues of $A$ below $t$ are characterised by the quantities
\begin{equation} \label{eq34}
  \mu_j^-\equiv\mu_j^-(t)=\min_{\substack {V\subset \dom(A)\\ \x{dim} V=j}} \; \max_{\substack {u\in V \\ u \neq 0}} \; \frac{\frak{a}^1_t(u,u)}{\frak{a}^2_t(u,u)} \qquad \forall j=1,\ldots,\ell(t)
\end{equation}
and those above $t$ are characterise by the quantities
\begin{equation} \label{eq34plus}
  \mu_j^+\equiv \mu_j^+(t)=\max_{\substack {V\subset \dom(A)\\ \x{dim} V=j}} \; \min_{\substack {u\in V \\ u \neq 0}} \; \frac{\frak{a}^1_t(u,u)}{\frak{a}^2_t(u,u)} \qquad \forall j\in \mathbb{N}.
\end{equation}
The following lemma is exactly \cite[Theorem~1.1]{1995Zimmer}.
Here and everywhere below the index $j$ will count multiplicities.

\begin{lem} \label{lem42}
Let $t\not\in \spec(A)$. Then
\[
     \mu_j^-(t)=\frac{1}{\lambda_{\ell(t)-j+1}-t} \qquad \qquad \forall j=1,\ldots,\ell(t)
\]
and
\[
      \mu_j^+(t)=\frac{1}{\lambda_{\ell(t)+j}-t} \qquad \qquad \forall j\in\mathbb{N} .
\]
\end{lem}
\begin{proof}
\begin{align*}
 \mu_j^-&=\min_{\substack {\mathcal{V}\subset \dom(A)\\ \x{dim} \mathcal{V}=j}} \; \max_{\substack {u\in \mathcal{V} \\ u \neq 0}} \; \frac{\frak{a}^1_t(u,u)}{\frak{a}^2_t(u,u)}\\
&=\min_{\substack {\mathcal{V}\subset \dom(A)\\ \x{dim} \mathcal{V}=j}} \max_{\substack {u\in \mathcal{V} \\ u \neq 0}} \frac{\langle (A-t)u,u\rangle }{\langle (A-t)u,(A-t)u\rangle }\\
&=\min_{\substack {\mathcal{W} \subset \mathcal{H} \\ \x{dim} \mathcal{W}=j}} \max_{\substack {v\in \mathcal{W}}} \frac{\langle v,(A-t)^{-1}v\rangle }{\langle v,v\rangle }.
\end{align*}
Here
\[
    \mathcal{V}=(A-t)^{-1}\mathcal{W},  \qquad \mathcal{W}=(A-t)\mathcal{V}
\]
and $\x{dim}\mathcal{V}=\x{dim}\mathcal{W}$ is guaranteed due to the invertibility of $(A-t)$.
Therefore $\mu_j^-<0$ is the eigenvalue of $(A-t)^{-1}$ which is on the $j$th position counting multiplicities right to left from $0$. As there are exactly $\ell$ eigenvalues of $A$ below $t$, according to the Spectral Mapping Theorem, then there are exactly $\ell$ of these $\mu_j^-$ which are negative. Due to the ordering of the eigenvalues of $A$   relative to $t$, the index $j$ on the $\mu^-_j$ corresponds to that of  $\lambda_{\ell-j+1}$. Thus
\[
\mu_j^-=\frac{1}{\lambda_{\ell-j+1}-t}.
\]

For the other case we proceed analogously, taking into account the fact that we have infinitely many eigenvalues of $A$ above $t$, which in turns become the positive $\mu^+_j$ accumulating at $0$ from the right.
\end{proof}

This statement motivates the following definition. Assume that  \eqref{eq31} has exactly $m^-\equiv m^-(t)$ negative eigenvalues and $m^+\equiv m^+(t)$ positive eigenvalues. We will see below that $m^-=\ell$ whenever $\mathcal{L}$ is sufficiently close to the eigenspace associated to $\{\lambda_1,\ldots,\lambda_\ell\}$. 
Let $\tau_j^-\equiv \tau_j^-(t)$ denote these negative eigenvalues and $\tau_j^+\equiv \tau_j^+(t)$ denote these positive eigenvalues, respectively, for $j=1,\ldots,m^\pm(t)$.

\begin{lem}      \label{finite}
For $t\not \in \spec(A)$,  
\begin{equation} \label{eq35}
 \tau^-_j(t)=\min_{\substack {\mathcal{V}\subset \mathcal{L} \\ \x{dim} \mathcal{V}=j}} \; \max_{\substack {u\in \mathcal{V} \\ u \neq 0}} \; \frac{\frak{a}^1_t(u,u)}{\frak{a}^2_t(u,u)} \qquad \forall j=1,\ldots,m^-(t)
\end{equation}
and
\begin{equation} \label{eq35plus}
 \tau^+_j(t)=\max_{\substack {\mathcal{V}\subset \mathcal{L} \\ \x{dim} \mathcal{V}=j}} \; \min_{\substack {u\in \mathcal{V} \\ u \neq 0}} \; \frac{\frak{a}^1_t(u,u)}{\frak{a}^2_t(u,u)} \qquad \forall j=1,\ldots,m^+(t).
\end{equation}
\end{lem}
\begin{proof}
This follows immediately from the Min-max Principle applied to the matrix problem associated to \eqref{eq31}.
\end{proof}

A combination of lemmas~\ref{finite} and \ref{lem42} leads to the following generalisation of Lemma~\ref{lem31}, which is exactly the content of \cite[Theorem~2.4]{1995Zimmer}. See also \cite[Theorem~11]{2004Davies} and \cite[Corollary~7]{2013Boulton}.

\begin{lem} \label{lem43}
 Let $t\not\in \spec(A)$. Then
\begin{equation}  \label{lower}
t+\frac{1}{\tau^-_j(t)}\leqslant \lambda_{\ell(t)-j+1} \qquad \forall j=1, \ldots ,m^-(t)
\end{equation}
and
\begin{equation}  \label{upper}
t+\frac{1}{\tau^+_j(t)}\geqslant \lambda_{\ell(t)+j} \qquad \forall j=1, \ldots ,m^+(t).
\end{equation}
\end{lem}
\begin{proof}
From \eqref{eq34} and \eqref{eq35} we immediately get
\[
  \tau^-_j \geqslant \mu_j^- \quad \text{so} \quad \frac{1}{\mu_j^-} \geqslant \frac{1}{\tau_j^-}.
\]
Then, according to Lemma~\ref{lem42},\[\lambda_{\ell-j+1} \geqslant t+ \frac{1}{\tau^-_j}.
\]
The proof of the other statement is very similar.
\end{proof}

In view of Lemma~\ref{lem43}, the inverse residuals $\tau^\pm_j$ give lower and upper bounds for the eigenvalues of $A$ which are below and above $t$, respectively.  

\begin{rem} \label{rem_new1}
In general $\ell(t)\geq m^-(t)$,
but it is not necessarily guaranteed that $\ell(t)= m^-(t)$.  On a practical setting, \emph{a priori} information about the value of $\ell(t)$ is required, if we wish to determine the correct indexing of the lower bounds for the points in $\spec(A)$ which are below $t$. This is a known limitation of the current approach, which in particular frameworks can be handled by means of homotopy methods \cite{1991Plum,1990Plum}. 
\end{rem}

In spite of this observation, note that a positive $m^-(t)$ implies the existence of 
eigenvalues of $A$ below $t$.


\section{Optimal choice of the shift} \label{impzimesec}
The quality of the bounds established in Lemma~\ref{lem43} depends on the choice of the parameter $t$. We now examine the optimality of these bounds and show that this is achieved as $t$ moves away from the eigenvalue of interest. We begin with an auxiliary statement.

\begin{lem} \label{lem51}
Let  $s,t\not\in \spec(A)$ be such that $t<s$. Assume that $u\in \mathcal{L}$ is such that either $\frak{a}^1_t(u,u)<0$ or  $\frak{a}^1_{s}(u,u)>0$. Then
\begin{equation} \label{eq51}
 t+\frac{\frak{a}^2_t(u,u)}{\frak{a}^1_t(u,u)} \leqslant s +\frac{\frak{a}^2_{s}(u,u)}{\frak{a}^1_{s}(u,u)}.
\end{equation}
\end{lem}
\begin{proof}
Let $u$ be as in the hypothesis. Without loss of generality we can assume that $\|u\|=1$. 
By the Cauchy-Schwarz inequality,
\begin{align*}
 \frak{a}^1_t(u,u)^2&=\Bigl|\langle(A-t) u,u\rangle \Bigr|^2\\
&\leqslant \langle(A-t) u,(A-t)u\rangle \langle u,u\rangle\\
&= \frak{a}^2_t(u,u).
\end{align*}
Then
\[
0 \leqslant (s-t)\left(\frak{a}^2_t(u,u)-\frak{a}^1_t(u,u)^2\right).
\]
Hence
\[
   \frak{a}^1_t(u,u)\frak{a}^2_t(u,u)-(s-t)\frak{a}^2_t(u,u) \leqslant \frak{a}^1_t(u,u) \frak{a}^2_t(u,u) -(s-t) \frak{a}^1_t(u,u)^2.
\]
According to the hypothesis, either both $\frak{a}^1_t(u,u)$ and $\frak{a}^1_{s}(u,u)$ are positive or both
are negative. Thus
\[
\frac{\frak{a}^2_t(u,u)}{\frak{a}^1_t(u,u)} \leqslant \frac{\frak{a}^2_t(u,u)-(s-t) \frak{a}^1_t(u,u)}{\frak{a}^1_t(u,u)-(s-t)\frak{a}^0(u,u)}=\frac{\frak{a}^2_{s}(u,u)}{\frak{a}^1_{s}(u,u)}+s-t.
\]
\end{proof}

The following is the main statement of this section. We formulate it in terms of $t\pm R$ for a fixed value of $t$ and consider moving $R$ along $(0,\infty)$. Note that the hypothesis ensures that $\ell(t \pm R)=\ell(t)$.

\begin{thm} \label{lem52}
 Let $t \in \RR$ and $R>0$ be such that \[[t-R,t+R]\cap \spec(A)=\varnothing.\]  Then the following holds true.
\begin{enumerate}
\item \label{sta1} $m^-(t)\leqslant m^-(t+R)$ and 
\[
     t+\frac{1}{\tau^-_j(t)} \leqslant t+R +\frac{1}{\tau^-_j(t+R)}\leqslant \lambda_{\ell(t)-j+1} \qquad \forall
     j=1,\dots,m^-(t)
\]
\item \label{sta2} $m^+(t)\leqslant m^+(t-R)$ and 
\[
    \lambda_{\ell(t-R)+j}\leqslant t-R+\frac{1}{\tau^+_j(t-R)} \leqslant t +\frac{1}{\tau^+_j(t)} 
    \qquad \forall
j=1,\dots,m^+(t)
\]
\end{enumerate}
\end{thm}
\begin{proof}
We include the proof of the statement ``\ref{sta1}.'' only. The statement ``\ref{sta2}.'' is shown in a similar fashion.  

Let $\mathcal{V}^- \subset \mathcal{L}$ with $\x{dim}\mathcal{V}^-=j$ be such that there exists $u_j^-\in  \mathcal{V}^-$ with $\frak{a}^1_t(u^-_j,u^-_j)<0$ and
\[
\tau^-_j(t)=\min_{\substack {\mathcal{V}\subset \mathcal{L} \\ \x{dim} \mathcal{V}=j}} \; \max_{\substack {u\in \mathcal{V} \\ u \neq 0}} \; \frac{\frak{a}^1_t(u,u)}{\frak{a}^2_t(u,u)}=  \max_{\substack {u\in \mathcal{V}^-}} \;\frac{\frak{a}^1_t(u,u)}{\frak{a}^2_t(u,u)} =\frac{\frak{a}^1_t(u^-_j,u^-_j)}{\frak{a}^2_t(u^-_j,u^-_j)}.
\]
Then 
\[
 \tau^-_j(t+R)=\min_{\substack {\mathcal{V}\subset \mathcal{L} \\ \x{dim} \mathcal{V}=j}} \; \max_{\substack {u\in \mathcal{V} \\ u \neq 0}} \; \frac{\frak{a}^1_{t+R}(u,u)}{\frak{a}^2_{t+R}(u,u)}  \leqslant \max_{\substack {u\in \mathcal{V}^-}} \; \frac{\frak{a}^1_{t+R}(u,u)}{\frak{a}^2_{t+R}(u,u)}
 =\frac{\frak{a}^1_{t+R}(v^-_j,v^-_j)}{\frak{a}^2_{t+R}(v^-_j,v^-_j)}
\]
for a special (maximising)  vector $v_j^-\in \mathcal{V}^-$.  

Since $\frak{a}^1_{t}(v_j^-,v_j^-)<0$ and
\[
    \frak{a}^1_{t+R}(v_j^-,v_j^-)=\frak{a}^1_{t}(v_j^-,v_j^-)-R\|v_j^-\|^2,
\]
then also $\frak{a}^1_{t+R}(v_j^-,v_j^-)<0$. Hence $m^-(t+R)\geqslant m^-(t)$.

Now, as both sides of the inequality above are negative, we gather that
\[
       \frac{1}{\tau^-_j(t+R)}\geqslant  \frac{\frak{a}^2_{t+R}(v^-_j,v^-_j)}{\frak{a}^1_{t+R}(v^-_j,v^-_j)}.
\]
Hence, by the definition of $u_j^-$ above, the fact that the fractions involved are negative, and an application of  \eqref{eq51} with $s=t+R$ and $u=v^-_j$,  we get
\begin{align*}
 t+\frac{1}{\tau^-_j(t)} &= t+ \frac{\frak{a}^2_t(u_j^-,u_j^-)}{\frak{a}^1_t(u_j^-,u_j^-)}   
     \leqslant  t+ \frac{\frak{a}^2_t(v_j^-,v_j^-)}{\frak{a}^1_t(v_j^-,v_j^-)}   \\ &  \leqslant t+R+ 
     \frac{\frak{a}^2_{t+R}(v_j^-,v_j^-)}{\frak{a}^1_{t+R}(v_j^-,v_j^-)}
\leqslant t+R +\frac{1}{\tau^-_j(t+R)}.
\end{align*}
Note that we can do all this for $j$ running from $1$ to $m^-(t)$.
\end{proof}

Consider the specific group of $m$ eigenvalues, 
\begin{equation}    \label{group}
   \{\lambda_{1}\leq\ldots\leq\lambda_{m} \}\subset \spec(A).
\end{equation}
Suppose we find upper bounds for these eigenvalues from a fixed $t^-<\lambda_{1}$ and 
lower bounds from a fixed $\lambda_{m}<t^+<\lambda_{m+1}$. Denote these bounds by
\[
     \bar{\lambda}_{j,\mathrm{up}}(t^-)= t^-+ \frac{1}{\tau^+_{j}(t^-)}
\qquad \text{and} \qquad
     \bar{\lambda}_{j,\mathrm{low}}(t^+)= t^+ + \frac{1}{\tau^-_{\ell(t^+)-j+1}(t^+)}.
\]
If we choose $\lambda_{m+1}>s^+\geq t^+$ and $s^-\leq t^-$, denoting
\[
\x{tol}(r^-,r^+,j)=\Bigl|\bar{\lambda}_{j,\mathrm{up}}(r^-)-\bar{\lambda}_{j,\mathrm{low}}(r^+)\Bigl|
\qquad  \forall j=1,\ldots,\ell \qquad r\in\{t,s\},
\]
gives
\[
\x{tol}(s^-,s^+,j) \leqslant \x{tol}(t^-,t^+,j).
\]
Indeed, set $s^+=t^++R^+$ and $s^-=t^--R^-$ for $R^\pm \geq 0$. 
According to Theorem~\ref{lem52}, 
\[
\bar{\lambda}_{j,\mathrm{low}}(t^+) \leqslant \bar{\lambda}_{j,\mathrm{low}}(s^+) < \lambda_j 
\qquad
\text{and}
\qquad
\lambda_j <\bar{\lambda}_{j,\mathrm{up}}(s^-) \leqslant \bar{\lambda}_{j,\mathrm{up}}(t^-).
\]
Then
\[
 |\bar{\lambda}_{j,\mathrm{up}}(s^-)-\bar{\lambda}_{j,\mathrm{low}}(s^+)| \leq |\bar{\lambda}_{j,\mathrm{up}}(t^-)-\bar{\lambda}_{j,\mathrm{low}}(t^+)|.
\]


\section{Convergence} \label{convergencezimmermann}

We now examine the convergence of the bounds established in Lemma~\ref{lem43}, in a regime where  $\mathcal{L}$
captures the eigenvectors of $A$. Our aim will be to show that, under suitable conditions,
\begin{equation}    \label{bounds}
      t+\frac{1}{\tau^-_j(t)} \uparrow \lambda_{\ell(t)-j+1}
      \qquad \text{and} \qquad
      t+\frac{1}{\tau^+_j(t)} \downarrow \lambda_{\ell(t)+j},
\end{equation}
for any $j=1,\ldots,m^{\pm}(t)$.

\subsection{Auxiliary results}
We firstly set a notation that simplifies greatly the arguments below. Let 
\begin{equation} \label{eq40}
\frak{b}^-_t(u,v)=\frak{a}^1_t(u,v)+(1-\mu^-_1(t))\frak{a}^2_t(u,v),
\end{equation}
and
\begin{equation}\label{eq40plus}
\frak{b}^+_t(u,v)=-\frak{a}^1_t(u,v)+(1+\mu^+_1(t))\frak{a}^2_t(u,v),
\end{equation}
(recall \eqref{eq34} and \eqref{eq34plus}). The quadratic forms \[\frak{b}^{\pm}_t: \dom(A)\times \dom(A)\longrightarrow \mathbb{C}\] are positive and closed. Moreover
\[
\frak{a}^2_t(u,u) \leqslant \frak{b}^{\pm}_t(u,u)  \qquad \forall u \in \dom(A).
\]
We write $^\pm\llbracket u\rrbracket_t=\frak{b}^{\pm}_t(u,u)^{1/2}$, so that $\left(\dom(A),\frak{b}^{\pm}_t \right)$ are a Hilbert spaces with respect to this norm. 

Let
\[
\nu_j^-
\equiv\nu^-_j(t)=\min_{\substack {\mathcal{V}\subset \dom(A) \\ \x{dim} \mathcal{V}=j}} \;\ \max_{u\in \mathcal{V}} \frac{\frak{b}_t^-(u,u)}{\frak{a}^2_t(u,u)}= \mu_j(t)^- +1 - \mu^-_1(t)>0
\]
and
\[
\nu_j^+
\equiv\nu^+_j(t)=\min_{\substack {\mathcal{V}\subset \dom(A) \\ \x{dim} \mathcal{V}=j}} \;\ \max_{u\in \mathcal{V}} \frac{\frak{b}_t^+(u,u)}{\frak{a}^2_t(u,u)}=- \mu_j(t)^+ +1 + \mu^+_1(t)>0.
\]
Let
\[
\varrho_j^-\equiv \varrho_j^-(t)=\min_{\substack {\mathcal{V}\subset \mathcal{L}  \\ \x{dim} \mathcal{V}=j}} \; \max_{u\in \mathcal{V}} \frac{\frak{b}^-_t(u,u)}{\frak{a}^2_t(u,u)}= \tau_j^-(t) +1 - \mu^-_1(t)>0
\]
and
\[
\varrho_j^+\equiv \varrho_j^+(t)=\min_{\substack {\mathcal{V}\subset \mathcal{L}  \\ \x{dim} \mathcal{V}=j}} \; \max_{u\in \mathcal{V}} \frac{\frak{b}^+_t(u,u)}{\frak{a}^2_t(u,u)}= -\tau_j^+(t) +1 + \mu^+_1(t)>0.
\]
As $|\tau_j^{\pm}-\mu^{\pm}_j|=|\varrho_j^{\pm}-\nu^{\pm}_j|$, in order to examine the limits
\eqref{bounds}, we might as well focus on studying the convergence $\varrho_j^{\pm} \downarrow \nu_j^{\pm}$. Our analysis below follows closely that of the classical setting \cite[\S4.4 and Ch.6]{1973Strang}. 

Let $P^{\pm}:\dom(A) \rightarrow \mathcal{L}$ be the projections orthogonal in the inner products $\frak{b}^{\pm}_t(u,v)$ respectively.  Let $\{\phi_k\}_{k=1}^\infty$ be a fixed family of eigenvectors such that $A\phi_k=\lambda_k\phi_k$. We assume that the $\phi_k$ are chosen to be orthonormal in the inner product $\frak{a}^2_t(\cdot,\cdot)$ of $\dom(A)$. 

Let
\[ 
    E_j= \x{Span} \{\phi_k\}_{k=1}^{j}
\]
be the eigenspace associated to the eigenvalues up to index $j$. Set
\[
 \qquad F_j=\{ u \in E_j : |u|_t=1\}
\]
and
\begin{equation} \label{eq41}
   \sigma_j^{\pm}\equiv \sigma_j^{\pm}(t)=\max_{u \in F_j} |2\frak{a}_t^2(u,u-P^{\pm}u)-\frak{a}_t^2(u-P^{\pm}u,u-P^{\pm}u)| .
\end{equation}
The following lemma mimics \cite[Lemma~6.1]{1973Strang}.

\begin{lem} \label{lem45}
If $\sigma_j^{\pm}(t) < 1$, then $\varrho_j^{\pm}(t) \leqslant \frac{\nu_j^{\pm}(t)}{1-\sigma^{\pm}_j(t)}$.
\end{lem}
\begin{proof}
Suppose $P^{\pm}\phi=0$ for some $\phi \in E_j$ such that $|\phi|_t=1$.
Then 
\[
\frak{a}^2_t(\phi,\phi-P^{\pm}\phi)=\frak{a}^2_t(\phi,\phi)=1 \quad \text{and} \quad \frak{a}^2_t(\phi-P^{\pm}\phi,\phi-P^{\pm}\phi)=1.
\]
If we substitute into \eqref{eq41}, we get
\[ 
\left|2\frak{a}^2_t(\phi,\phi-P^{\pm}\phi)-\frak{a}^2_t(\phi-P^{\pm}\phi,\phi-P^{\pm}\phi)\right|=1.
\]
Hence
\[
\sigma^{\pm}_j=\max_{u \in F_j} \left|2\frak{a}^2_t(u,u-P^{\pm}u)-\frak{a}^2_t(u-P^{\pm}u,u-P^{\pm}u)\right|\geqslant 1.
\]

The above shows that necessarily $P^{\pm}\phi\not=0$ for any $\phi \in F_j$, if the hypothesis is to be satisfied. In turns this implies that $ \x{dim} P^{\pm} E_j=j$. Then,
\[
\varrho_j^{\pm} \leqslant \max_{v \in P^{\pm} E_j} \frac{^\pm\llbracket v \rrbracket_t^2}{|v|_t^2}=\max_{u \in F_j} \frac{^\pm\llbracket P^{\pm}u \rrbracket_t^2}{|P^{\pm}u|_t^2}.
\]
Now since $P^{\pm}$ is the orthonormal projection in the inner product $\frak{b}_t^{\pm}$, then $^\pm\llbracket P^{\pm}u \rrbracket_t \; \leqslant \; ^\pm\llbracket u\rrbracket_t$.  Also
\[
|P^{\pm}u|_t^2=\frak{a}^2_t(u,u)-2\frak{a}^2_t(u,u-P^{\pm}u)+\frak{a}^2_t(u-P^{\pm}u,u-P^{\pm}u) \geqslant 1-\sigma_j^{\pm}.
\]
So
\[
\varrho_j^{\pm} \leqslant \max_{u\in F_j}\frac{^\pm\llbracket u\rrbracket_t^2}{1-\sigma^{\pm}_j}=\frac{\nu^{\pm}_j}{1-\sigma^{\pm}_j}.
\]
\end{proof}

By virtue of the previous lemma and the fact that 
\[
 0<\nu_j^{\pm}\leq \varrho_j^{\pm},
\]
if $\sigma_j^{\pm}$ is close to 0, then $\varrho_j^{\pm}$ is close to $\nu_j^{\pm}$. In the next two lemmas we set conditions for $\sigma_j^{\pm}$ to be small.

\begin{lem} \label{lem46}
 Let $u \in F_j$. Then
\[
|2\frak{a}^2_t(u,u-P^{\pm}u)|\leqslant 2 {\displaystyle \sum_{k=1}^j} {^\pm\!\left\llbracket \phi_k-P^{\pm}\phi_k\right\rrbracket ^2_t}.
\]
\end{lem}
\begin{proof} We only consider the proof with the ``$-$'' sign, the proof with the ``$+$'' sign being analogous.

 For $v\in \dom(A)$  and  $k \leqslant j$,
\begin{align*}
 \frak{b}_t^-\left(\phi_k-P^- \phi_k,v-P^-v\right)&=\frak{b}_t^-\left(\phi_k,v-P^-v\right)\\
&=\frak{a}^1_t\left(\phi_k,v-P^-v\right)+\left(1-\mu^-_1\right)\frak{a}^2_t\left(\phi_k,v-P^-v\right)\\
&=\left(\frac{1}{\lambda_{k}-t}+1-\mu^-_1\right) \frak{a}^2_t\left(\phi_k,v-P^-v \right).
\end{align*}
Then
\[ 
\frak{a}^2_t(\phi_k,v-P^-v)=(1+\mu^-_k-\mu^-_1)^{-1}\frak{b}_t^-(\phi_k-P^-\phi_k,v-P^-v).
\]
So, if we expand $u={\displaystyle \sum_{k=1}^j}c_k\phi_k$ where ${\displaystyle \sum_{k=1}^j}|c_k|^2=1$ (recall that we are assuming $|\phi_k|_t=1$),
then
\begin{align*}
 2\Bigl|\frak{a}^2_t(u,&u-P^-u)\Bigr|  \\
  &= 2\Bigl|\frak{b}_t^-\Bigl({\displaystyle\sum_{k=1}^j} c_k(1+\mu^-_k-\mu^-_1)^{-1}(\phi_k-P^-\phi_k),{\displaystyle\sum_{k=1}^j} c_k(\phi_k-P^-\phi_k)\Bigr)\Bigr|\\
&\leqslant 2 \leftidx{^-}{\left\llbracket {\displaystyle\sum_{k=1}^j}c_k(1+\mu^-_k-\mu^-_1)^{-1}(\phi_k-P^-\phi_k)\right\rrbracket}{_t}  \; \leftidx{^-}{\left\llbracket {\displaystyle\sum_{k=1}^j} c_k(\phi_k-P^-\phi_k)\right\rrbracket}{_t}\\
&\leqslant 2\Biggl({\displaystyle\sum_{k=1}^j}|c_k|\left(|1+\mu^-_k-\mu^-_1|^{-1}\right)\leftidx{^-}{\llbracket \phi_k-P^-\phi_k\rrbracket}{_t}\Biggr)\Biggl({\displaystyle\sum_{k=1}^j}|c_k| \; \leftidx{^-}{\llbracket \phi_k-P^-\phi_k\rrbracket}{_t}\Biggr)\\
&\leqslant 2 \Biggl({\displaystyle\sum_{k=1}^j}|c_k|^{2}(1+\mu^-_k-\mu^-_1)^{-2}\Biggr)^{1/2}\Biggl({\displaystyle\sum_{k=1}^j}|c_k|^2\Biggr)^{1/2}\Biggl({\displaystyle\sum^j_{k=1}} \leftidx{^-}{\llbracket \phi_k-P^-\phi_k\rrbracket}{^2_t}\biggr)\\
&\leqslant 2{\displaystyle\sum_{k=1}^j} \leftidx{^-}{\llbracket \phi_k-P^-\phi_k\rrbracket}{^2_t}.
\end{align*}
The second and fourth inequalities are the Cauchy-Schwarz inequality, the third inequality is the triangle inequality. Here we are using the fact that $1+\mu^-_k-\mu^-_1>1$, so that
\[
 \frac{1}{(1+\mu^-_k-\mu^-_1)^{2}}<1.
\]
\end{proof}

\begin{lem} \label{lem47}
 Let $u \in F_j$. Then
\[
\frak{a}^2_t(u-P^{\pm}u,u-P^{\pm}u)\leqslant {\displaystyle\sum_{k=1}^j}\leftidx{^\pm}{\llbracket \phi_k-P^{\pm}\phi_k \rrbracket}{^2_t}. 
\]
\end{lem}
\begin{proof}
Expand $u={\displaystyle \sum_{k=1}^j}c_k\phi_k$ for ${\displaystyle \sum_{k=1}^j}|c_k|^2=1$.
Then, 
\begin{align*}
 \frak{a}^2_t(u-P^{\pm}u,u-P^{\pm}u) &\leqslant \frak{b}_t^\pm(u-P^{\pm}u,u-P^{\pm}u)=\leftidx{^\pm}{\llbracket u-P^{\pm}u\rrbracket}{^2_t}\\
&\leqslant \Biggl({\displaystyle \sum_{k=1}^j}|c_k| \; ^\pm\llbracket \phi_k-P^{\pm}\phi_k\rrbracket_t\Biggr)^2 \\
& \leqslant \Biggl({\displaystyle \sum_{k=1}^j}|c_k|^2\Biggr)\Biggl({\displaystyle \sum_{k=1}^j}\leftidx{^\pm}{\llbracket \phi_k-P^{\pm}\phi_k\rrbracket}{^2_t}\Biggr)\\
&\leqslant {\displaystyle\sum_{k=1}^j}\leftidx{^\pm}{\llbracket \phi_k-P^{\pm}\phi_k\rrbracket}{^2_t}.
\end{align*}
\end{proof}

According to lemmas~\ref{lem46} and \ref{lem47}, it follows that 
\begin{equation*}
 \sigma_j^{\pm}(t) \leqslant 3 {\displaystyle\sum_{k=1}^j} \leftidx{^\pm}{\llbracket \phi_k-P^{\pm}\phi_k\rrbracket}{^2_t}.
\end{equation*}
The next statement is the main result of this section. Similar results in the context of the quadratic method have been studied in \cite[Section 4]{2013Boulton}.

\subsection{Main statement on convergence}
\begin{thm} \label{pro63}
Fix a group of $m$ eigenvalues as in \eqref{group} and let \[t\in (-\infty,\lambda_{m+1})\setminus \spec A.\] There exist a constant $C_t>0$ only dependant on
$A$, $t$ and $m$, ensuring the validity of the following. If $\mathcal{L}\subset \dom(A)$ is such that, 
for any $k=1,\ldots,m$ there exist $v_k\in \mathcal{L}$ such that
\[
      ^\pm\llbracket \phi_k - v_k \rrbracket_t \leq \varepsilon_k
\]
where
\[
       \left(\sum_{k=1}^m \varepsilon_k^2\right)^{1/2} \leq \varepsilon <\frac{1}{\sqrt{6}},
\]
then
\begin{equation}   \label{convzmminus}
         0<\lambda_{\ell(t) -k+1} -\left( t+\frac{1}{\tau_k^-(t)}  \right) \leq \frac{C_t}{|\tau_k^-(t)|} \varepsilon^2
\qquad \forall k=1,\ldots,m^-(t)
\end{equation}
and
\begin{equation} \label{convzmplus}
         \begin{aligned} 0< \left( t+\frac{1}{\tau_k^+(t)}\right) - \lambda_{\ell(t) +k}  &\leq \frac{C_t}{|\tau_k^+(t)|} \varepsilon^2
\\ & \forall k=1,\ldots,\min\{m^+(t),m-\ell(t)\}.
\end{aligned}
\end{equation}
\end{thm}
\begin{proof}
Assume that $\mathcal{L}$ is as in the hypothesis of the theorem. Then, Lemma~\ref{lem46} and Lemma~\ref{lem47} imply
\begin{equation} \label{eq42}
 \sigma_k^{\pm} \leqslant 3 \varepsilon^2 <\frac12.
\end{equation}
By  Lemma~\ref{lem45}, \eqref{eq42} implies,
\[
\nu_k^{\pm} \leqslant \varrho_k^{\pm} \leqslant \frac{\nu_k^{\pm}}{1-\varepsilon^2} .
\]
Thus 
\begin{align*}
 \varrho_k^{\pm} -\nu_k^{\pm} &\leqslant \frac{\nu_k^{\pm}}{1-\varepsilon^2}-\nu_k^{\pm} =\frac{\varepsilon^2}{1-\varepsilon^2} \; \nu_k^{\pm} \leqslant \frac{6}{5} \nu_k^{\pm} \varepsilon^2 .
\end{align*}

Let us show the existence of $C_t>0$ ensuring \eqref{convzmminus}. That we can find a (perhaps larger) $C_t>0$ ensuring both 
\eqref{convzmminus} and \eqref{convzmplus}, follows by means of the same technique applied to the ``+'' case.   Firstly note that
\begin{align*}
 \tau^-_k-\mu^-_k&=\varrho_k^--\nu_k^- \leqslant \frac65 \nu_k^-\varepsilon^2\\
 &\leq \frac65(\mu_k^-+1- \mu_1^-) \varepsilon^2.
\end{align*}
This yields
\[
\tau^-_k- \frac{1}{\lambda_{\ell-k+1}-t} \leqslant  \frac{6}{5} \left(\frac{1}{\lambda_{\ell-k+1}-t}-\frac{1}{\lambda_{\ell}-t}+1\right)\varepsilon^2.
\]
Since
\[
\lambda_{\ell-k+1}-\Bigl(t+\frac{1}{\tau^-_k}\Bigr)= \frac{\lambda_{\ell-k+1}-t}{\tau^-_k}\; \Bigl(\tau^-_k-\frac{1}{\lambda_{\ell-k+1}-t}\Bigr),
\]
then
\[
\lambda_{\ell-k+1}-\Bigl(t+\frac{1}{\tau^-_k}\Bigr) \leqslant \frac{C_t}{|\tau_k^-|}  \varepsilon^2,
\]
for a chosen constant that allows
\[
 C_t>\frac{6}{5}  \Biggl|1-\frac{\lambda_{\ell-k+1}-t}{\lambda_{\ell}-t}+\Bigl(\lambda_{\ell-k+1}-t\Bigr)\Biggr|
\]
for all $k=1,\ldots,m^-(t)$. This choice ensures \eqref{convzmminus}.

The arguments leading to the conclusion \eqref{convzmplus} are very similar.
\end{proof}

In the remaining parts of this paper we explore an implementation of the method introduced in Section~\ref{zimesec} to one-dimensional Schr\"odinger hamiltonians under suitable conditions on the potential. These operators satisfy the hypothesis of being bounded below  with a compact resolvent. 
The trial spaces will be taken to be generated by the finite element method. We  examine the convergence properties of this implementation, in the contexts of Theorem~\ref{pro63}.


\section{Eigenvalue bounds for Schr{\"o}dinger operators in one dimension} \label{zm_numerical}
Let $A=H$ be a one-dimensional semi-definite Schr{\"o}dinger operator with a compact resolvent. Let the trial subspaces $\mathcal{L}$ be constructed via the finite element method on a large, but finite, segment. Under standard assumptions on  the finite element spaces, below we determine the precise rate at which the upper and lower bounds in Lemma~\ref{lem43} converge to the true eigenvalues of $H$, as the mesh refines and the length of the segment grows. We refer to
\cite[Section~4]{20133Boulton} for similar results in the context of the so-called quadratic method.

\subsection{The model hamiltonian and its truncation to a finite box}
Let
\begin{equation*} \label{eq61}
 H u(x)=-u''(x)+V(x) u(x)  \qquad x\in (-\infty,\infty)
\end{equation*}
acting on $L^2(\mathbb{R})$. We assume that the potential $V(x)$ is real-valued, continuous and $V(x) \to \infty$  as $|x| \to \infty$. These conditions ensure that the operator $H$ is self-adjoint on a suitable domain and it has a compact resolvent \cite[Theorem XIII.67]{1980Barry}.
The domain of closure of the quadratic form associated to $H$ is
\[
    \dom (\frak{h}^{1}) = W^{1,2}(\mathbb{R}) \cap \{u\in L^2(\mathbb{R}):\||V|^{1/2}u\|<\infty\}.
\]
Note that this is the intersection of the maximal domains of the momentum operator and the operator of multiplication by $|V|^{1/2}$. Here and below we denote $\frak{h}\equiv\frak{a}$ for $A=H$.

The conditions on the potential imply that $V(x) \geqslant b_0>-\infty$ for all $x \in \mathbb{R}$ and a suitable constant  $b_0 \in \mathbb{R}$. This ensures that $H$ is bounded below, and $H \geq b_0$. Without loss of generality we will assume below that $b_0>0$.

From the fact that we are in one space dimension alongside with the condition of continuity on the potential, we know that all the eigenvalues of $H$ have multiplicity equal to 1. Moreover, the eigenfunctions are $C^\infty$ and they decay exponentially fast at infinity \cite[Section C.3]{1982Barry}. 
Following the notation for the generic operator $A$ above,  we write
\[
    \x{Spec}(H)=\{\lambda_1< \lambda_2<\ldots \}\subset(0,\infty).
\]
We let the orthonormal basis $\{\psi_j\}_{j=1}^\infty$ of $L^2(\mathbb{R})$ be such that \[H\psi_j=\lambda_j\psi_j.\] Without further mention,  below we often suppress the index $j$ from the eigenvalue and the eigenfunction, when it is sufficiently clear from the context.

Let $L > 0$. Consider the restricted operator
\begin{equation*} \label{eq62}
  H_L u(x)=-u''(x)+V(x) u(x)    \qquad x\in (-L,L),
\end{equation*}
subject to Dirichlet boundary conditions:   $u(-L)=u(L)=0$. As $L\to \infty$, we expect that the spectrum of $H_L$ approaches the spectrum of $H$. In fact, this turns out to happen exponentially fast (in $L$) for individual eigenvalues, whenever $V(x)$ is such that for every $b>0$ there exists a constant $k_b>0$ ensuring
\[
    |V(x)| \leqslant k_b \mathrm{e}^{b|x|}\qquad \forall x \in \mathbb{R}
\]
(see e.g. \cite[Theorem 4.4]{20133Boulton}). Without further mention, everywhere below we impose this additional assumption on the potential. 

The operator $H_L$ is self-adjoint on a domain defined via Friedrichs' extensions. Denote by $\frak{h}^{1,L}$ the quadratic form associated to $H_L$. Since $V$ is continuous on the whole of $\mathbb{R}$, then the domain of closure of  $\frak{h}^{1,L}$ is
\[
 \dom(\frak{h}^{1,L})=W^{1,2}_0(-L,L)
\]
see \cite[Theorem~VI.2.23 and VI.4.2]{kato}. 

As $b_0>0$, the forms $\frak{h}^{1}$ and $\frak{h}^{1,L}$ are both positive definite. Hence the quantities
\[
     \frak{h}^{1}(u,u)^{1/2} \qquad \text{and} \qquad   \frak{h}^{1,L}(u,u)^{1/2}
\]
define norms in $\dom(\frak{h}^{1})$ and $\dom(\frak{h}^{1,L})$ respectively.

\subsection{Finite element discretisation}
Let $\Xi$ be an equidistant partition of $[-L,L]$ into $n$ sub-intervals $I_l=[x_{l-1},x_l]$ of length $h=\frac{2L}{n}$. Let $\mathcal{L}_L^h\equiv \mathcal{L}_L^h(k,r)=V_h(k,r,\Xi)$ where 
\begin{equation}    \label{fespaces}
 V_h(k,r,\Xi)=\left\{v\in C^k(-L,L): \begin{aligned}& v \upharpoonright_{ I_l} \in \mathbb{P}_r(I_l)  \qquad 1 \leq l \leq n \\ & v(-L)=v(L)=0 
 \end{aligned} \right\}
\end{equation}
is the finite element space generated by $C^k$ conforming elements of order $r$ subject to Dirichlet boundary conditions. Here we require $k\geq 1$ and $r \geq 3$, so that $\mathcal{L}_L^h \subset \dom(\frak{h}^{2,L})$. Everywhere below we assume that these two parameters are fixed.

Let us now establish a concrete result showing that the hypothesis of Theorem~\ref{pro63}
is satisfied in the present setting. In turns this will imply that, if we implement a numerical strategy for computing $\tau_j^\pm(t)$, this implementation is convergent. 

Our first statement is \cite[Theorem~4.5]{20133Boulton}, but we include its proof for completeness.

\begin{lem} \label{thm62}
Fix $j\in \mathbb{N}$.
There exist $L_0 > 0$ large enough and $h_0>0$ small enough, such that the following is satisfied. For $L>L_0$ and $h<h_0$, we can always find $u_j \in \mathcal{L}_L^h$ such that 
\begin{enumerate}
 \item \label{c1} $\langle u_j-\psi_j, u_j-\psi_j \rangle \leq \epsilon^0(h,L)$
  \item \label{c2} $\langle H(u_j-\psi_j), u_j-\psi_j \rangle \leq \epsilon^1(h,L)$
  \item \label{c3} $\langle H(u_j-\psi_j), H(u_j-\psi_j) \rangle \leq \epsilon^2(h,L)$
\end{enumerate}
where 
\begin{align*}
\epsilon^0(h,L)&=c_{10} \mathrm{e}^{-4aL}+c_{20} h^{2(r+1)} \\
\epsilon^1(h,L)&=c_{11} \mathrm{e}^{-2aL}+c_{21} h^{2r} \\
\epsilon^2(h,L)&=c_{12}\mathrm{e}^{-2aL}+c_{22}h^{2(r-1)}
\end{align*}
for a suitable $a>0$ and constants $c_{nk}>0$ dependant on $j$, but independents of $L$ or $h$.
\end{lem}
\begin{proof}
Below we repeatedly use the estimate
\[
    \|v-v_h\|_{W^{p,2}(-L,L)} \leq \tilde{c} h^{r+1-p}
\]
where $v_h\in V_h(k,r,\Xi)$ is the interpolate of $v\in C^k\cap W_{\Xi}^{r+1,2}(-L,L)$.
See \cite[Theorem~3.1.6]{1978Ciarlet}. We set $u=\psi_h$.

For the property~\ref{c1}, observe that
\begin{align*}
  \langle u-\psi, u-\psi \rangle&=\int_{-\infty}^{\infty} |u-\psi|^2\dx \\
&=\int_{-\infty}^{-L}\!\!\!+\int_{L}^{\infty} |u-\psi|^2\dx +\int_{-L}^{L}|u-\psi|^2\dx \\
&=\int_{-\infty}^{-L}\!\!\!+\int_{L}^{\infty} |\psi|^2\dx +\int_{-L}^{L}|u-\psi|^2\dx \\
&\leq c_{10} e^{-4aL}+c_{20} h^{2(r+1)}.
\end{align*}

For the property~\ref{c2}, observe that
\begin{align*}
&\langle H(u-\psi), u-\psi \rangle\leq \int_{-\infty}^{\infty} \left(|u'-\psi'|^2+|V(x)||u-\psi|^2\right)\dx \\
&=\int_{-\infty}^{-L}\!\!\!+\int_{L}^{\infty}|\psi'|^2\dx +\int_{-\infty}^{-L}\!\!\!+\int_{L}^{\infty}|V(x)||\psi|^2\dx 
+ \int_{-L}^{L}|u'-\psi'|^2\dx  \\
& \qquad +\int_{-L}^{L}|V(x)||u-\psi|^2\dx \\
&\leq c_{31} \mathrm{e}^{-2aL}+c_{32} \mathrm{e}^{- \frac72aL}+c_{33} h^{2r} +c_{34} h^{2(r+1)}.
\end{align*}

For the property~\ref{c3}, observe that
\begin{align*}
&\langle H(u-\psi),H(u-\psi)\rangle\\
&\leq \int_{-\infty}^{\infty}|u''-\psi''|^2\dx +\int_{-\infty}^{\infty}|V(x)|^2|u-\psi|^2\dx \\
&\qquad +2\left[\left(\int_{-\infty}^{\infty}|V(x)|^2|u-\psi|^2\dx  \right)^{1/2} \left(\int_{-\infty}^{\infty}|u''-\psi''|^2\dx \right)^{1/2}\right]\\
&\leq\left( c_{41}\mathrm{e}^{-2aL}+c_{42}h^{2(r-1)}\right)+\left( c_{43}\mathrm{e}^{-3aL}+c_{44}h^{2(r+1)}\right)\\
&\qquad +\left[\left(c_{45}\mathrm{e}^{-3aL}+c_{46}h^{2(r+1)}\right)^{1/2} \left( c_{47}\mathrm{e}^{-2aL}+c_{48}h^{2(r-1)} \right)^{1/2}\right]\\
&= c_{41}\mathrm{e}^{-2aL}+c_{42}h^{2(r-1)}+ O(\mathrm{e}^{-3aL})+O(h^{2r}).
\end{align*}
Here we employ the Cauchy-Schwarz inequality and Newton's Generalised Binomial Theorem, as well as 
the exponential decay of the eigenfunctions (see e.g. \cite[lemmas~4.1 and 4.2]{20133Boulton}).
\end{proof}

\begin{col} \label{cor61}
Let $j\in\mathbb{N}$ and $t\not \in \spec(H)$ be fixed. There exist $\tilde{c}>0$ and $L_0 >0$ large enough, and $a>0$ and $h_0>0$ small enough, ensuring the following. There always exists $u_j\in \mathcal{L}_L^h$ such that
\begin{equation*}
 ^\pm\llbracket u_j-\psi_j\rrbracket_t \leqslant \tilde{c} (h^{r-1}+e^{-aL}) 
\end{equation*}
for $L>L_0$ and $h<h_0$.
\end{col}
\begin{proof}
Let $u_j\in\mathcal{L}_L^h$ be as in the previous lemma. Then
\begin{align*}
\frak{b}_t(u_j-\psi_j,u-\psi_j)&=\frak{h}^1_t(u_j-\psi_j,u_j-\psi_j)+\alpha(t)\;\frak{h}^2_t(u_j-\psi_j,u_j-\psi_j)\\
&=\langle H(u_j-\psi_j), u_j-\psi_j \rangle-t\langle u_j-\psi_j, u_j-\psi_j \rangle\\
&+\alpha(t) \big[\langle H(u_j-\psi_j), H(u_j-\psi_j) \rangle \\ 
&-2t\langle H(u_j-\psi_j),u_j-\psi_j \rangle+t^2\langle u_j-\psi_j, u_j-\psi_j \rangle \big]\\
&\leq \epsilon^1(h,L)-t\epsilon^0(h,L)+\alpha(t) \big[\epsilon^2(h,L)\\
& -2t\epsilon^1(h,L)+t^2 \epsilon^0(h,L) \big].
\end{align*}
This ensures the existence of the required constants.\end{proof}

\begin{thm} \label{cor62}
 Let $t\in (-\infty,\lambda_{m+1})\setminus \spec(H)$. Let $a>0$ be as in the previous corollary. There exist constants $C^\pm_{t,m}>0$, such that
\[
         0<\lambda_{\ell(t) -k+1} -\left( t+\frac{1}{\tau_k^-(t)}  \right) \leq C^-_{t,m}(h^{2(r-1)}+e^{-2aL})
\qquad \forall k=1,\ldots,m^-(t)
\]
and
\[
         \begin{aligned} 0< \left( t+\frac{1}{\tau_k^+(t)}\right) - \lambda_{\ell(t) +k}  &\leq  C^+_{t,m} (h^{2(r-1)}+e^{-2aL})
\\ & \forall k=1,\ldots,\min\{m^+(t),m-\ell(t)\}. \end{aligned}
\]
\end{thm}
\begin{proof}
The proof is a direct consequence of combining Theorem~\ref{pro63} with Corollary~\ref{cor61}.
\end{proof}

In the next section we consider two  models with concrete potentials and explore numerically the scope of this theorem.


\section{Numerical experiments} \label{numex}
In all the examples below we consider Hermite elements of order $r=3$ which are $C^1$ conforming, so $\mathcal{L}_L^h$ are as in \eqref{fespaces} with $v\!\upharpoonright_{I_l}$ a linear combination of two basis polynomials of order 3. 

As for model hamiltonians, we consider the quantum harmonic and anharmonic oscillators. 
Let $V(x)=x^2$ and $H^{\x{har}}\equiv H$. The exact eigenvalues are 
\[
     \x{Spec} (H^{\x{har}}) = \{2j+1: j=0,1,\ldots\}
\]
and the corresponding  eigenfunctions are
\[
    \psi_j(x)=h_j(x)e^{-x^2/2}
\]
where $h_j$ are the Hermite polynomial of order $j$.  
Let $H^{\x{anh}}=H$ for $V(x)=x^4$. In this case the exact eigenvalues and eigenfunctions
 are not known explicitly.

\subsection{Eigenvalue bounds and order of convergence}
Theorem~\ref{cor62} shows that the limits \eqref{bounds} are convergent at a rate proportional to $h^{2(r-1)}+e^{-2aL}$ in the regime $L\to \infty$ and $h\to 0$, when the trial spaces are chosen to be $\mathcal{L}=\mathcal{L}_L^h$. In all calculation below we fix $L=6$. 

The eigenvalue problem \eqref{eq31} is equivalent to that of the matrix problem
\begin{equation*}
 \tau \; \mathbf{A}_{2,t} = \mathbf{A}_{1,t}.
\end{equation*}
 In the numerical computations presently conducted, we have found $\tau$ from the solution of this linear eigenvalue problem.
The coefficients of the matrices $\mathbf{A}_0$, $\mathbf{A}_1$ and $\mathbf{A}_2$, were all computed 
analytically. See the Appendix in \cite{20133Boulton}.

Table \ref{table8} shows approximation of five eigenvalues of $H_L^{\x{har}}$ and $H_L^{\x{anh}}$ with $n=400$,
by means of the Galerkin method and by means of the Lehmann-Maehly-Goerisch method. The trial subspace has been taken exactly the same in both cases. The Galerkin method give certified upper bounds for $\lambda_j$. Observe that these upper bounds are roughly four orders of magnitude sharper than those obtained from an application of the   Lehmann-Maehly-Goerisch method with $t=-20$. See Figure~\ref{fig_con}.

\begin{rem} \label{rem_new2}
Recall Remark~\ref{rem_new1}. There is a strong indication that the numerical quantities shown in Table~\ref{table8} correspond to the first five eigenvalues of $H^{\x{anh}}$. This is certainly the case for 
$H^{\x{har}}$. However, without any \emph{a priori} information about 
the position of the sixth eigenvalue (not shown), this claim cannot be made with mathematical rigour. In the numerical experiments of this section we have  ``abused'' slightly the notation and written the index as $j$ for the eigenvalues of both $H^{\x{har}}$ and $H^{\x{anh}}$.  
\end{rem}

\begin{table}[t]
\centerline{\begin{tabular}{ |c |l| l| }
  \hline
  $j$ &  harmonic &  anharmonic \\ \hline
  1 & $1.00000000000018$ & $1.06036209048485$ \\ &  $^{1.00000000274037}_{0.99999999402733}$ & $1.060362^{10271726}_{05784546}$  \\ \hline
  2 & $3.00000000000167$ & $3.79967302981065$ \\
  & $^{3.00000004427331}_{2.99999993414717}$ & $3.79967^{336336235}_{266822753}$ \\ \hline
  3 & $5.00000000001386$ & $7.45569793805316$\\
   & $^{5.00000020285501}_{4.99999969518901}$ & $7.455^{70118941346}_{69223027522}$ \\ \hline
  4 & $7.00000000018134$ & $11.64474551167977$ \\
  & $^{7.00000063700910}_{6.99999905084962}$ & $11.6447^{5785577082}_{3508597596}$  \\ \hline
  5 & $9.00000000261104$ & $16.26182601985996$\\ & $^{9.00000158165711}_{8.99999763441928}$ & $16.261^{88816616746}_{79878260359}$ \\ \hline
  \end{tabular}
}
\caption{Upper bounds and approximated enclosures for five eigenvalues $\lambda_{j}$ of $H_6^{\x{har}}$ and $H_6^{\x{anh}}$. Here we have fixed $n=400$. In each row, the quantities on the top correspond to upper bounds determined by means of the Galerkin method, while the enclosures at the bottom correspond to those found by the method describe in Section~\ref{zimesec}. For the latter, the upper bounds were found by fixing $t=-20$ and the lower bounds were found by fixing $t=20$. \label{table8}}
\end{table}

Figure~\ref{fig8} shows $n$ versus the size of the eigenvalue enclosure $\lambda_{j,\x{up}}-\lambda_{j,\x{low}}$, for $\lambda_{j,\x{up}}>\lambda_j$ an upper bound found from \eqref{upper} and $\lambda_{j,\x{low}}<\lambda_j$ a lower bound found from \eqref{lower}. In this figure the slopes are fairly close to $4$, indicating that 
the convergence rates established in Theorem~\ref{cor62} are optimal.

\begin{figure}[t]
\centering
\includegraphics[width=4.5cm]{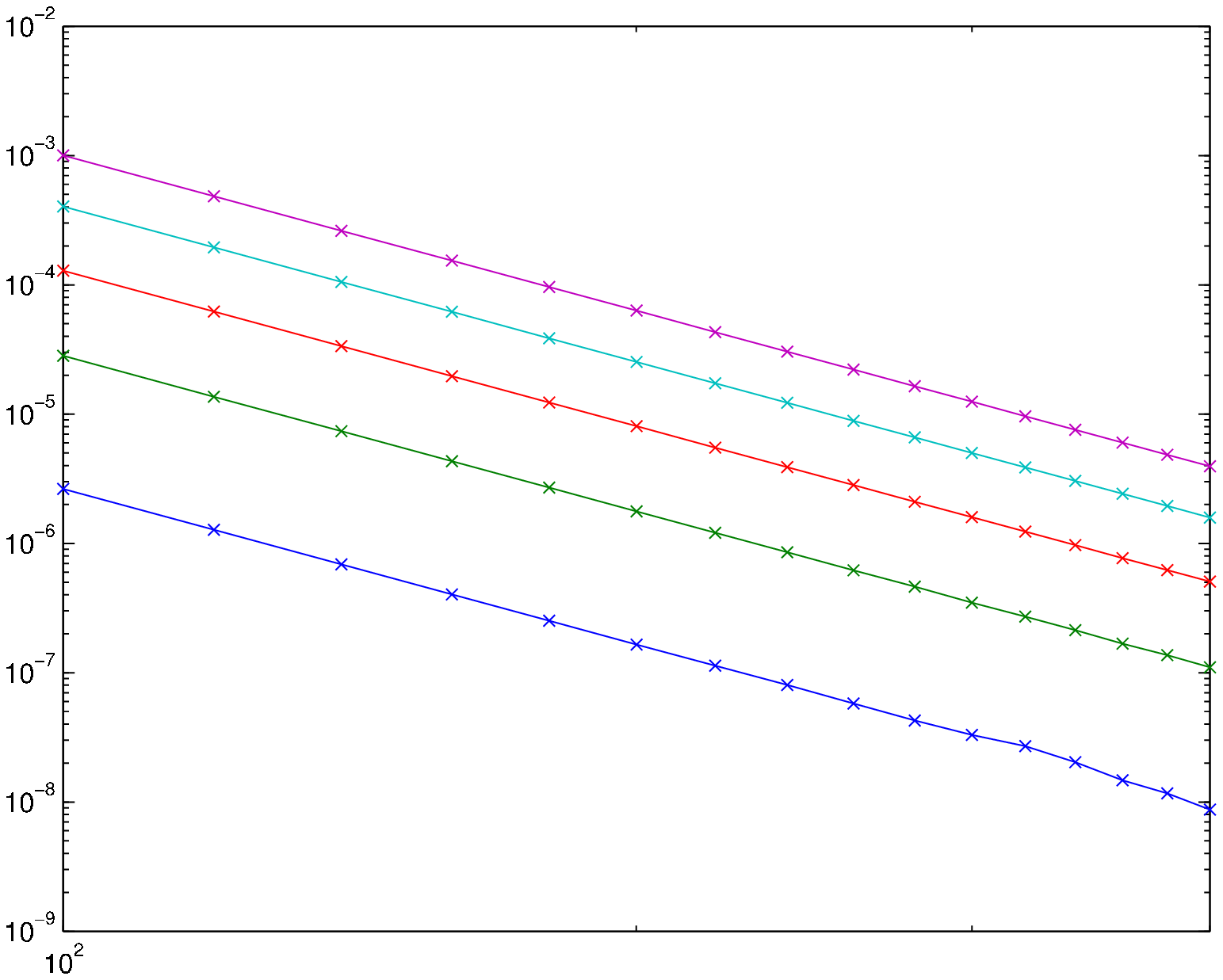}
\includegraphics[width=4.5cm]{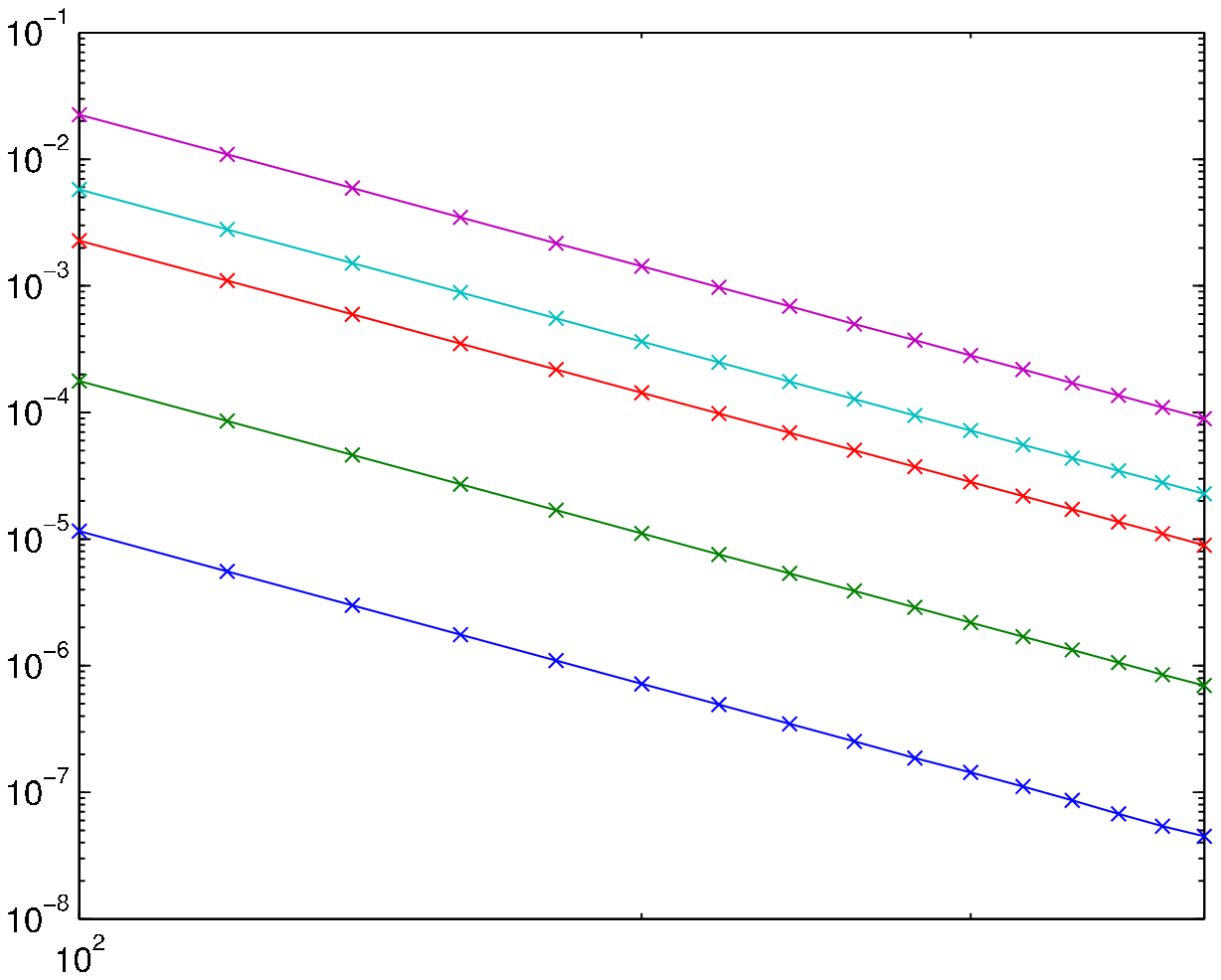}
\includegraphics[width=2cm]{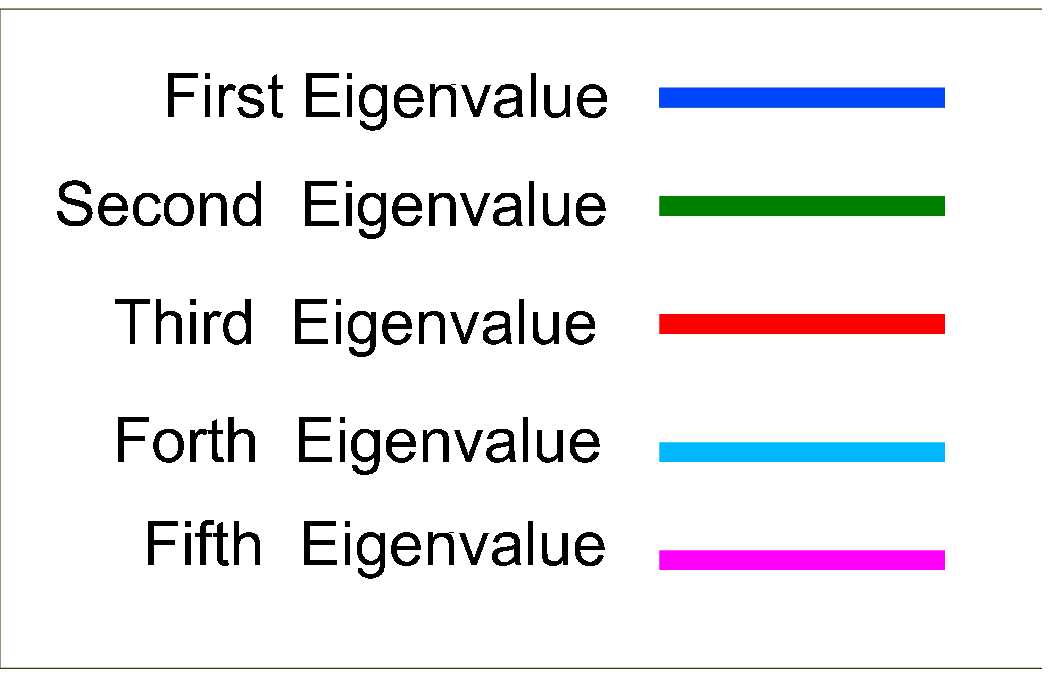}
\caption{Loglog plot of the residual $\lambda_{j,\x{up}}-\lambda_{j,\x{low}}$ as $n$ increases for $H_6^{\x{har}}$ (left) and $H_6^{\x{anh}}$ (right). Upper bounds are found by fixing $t=-20$ and lower bounds are found by fixing $t=20$.
Here $n$ runs from 100 to 400.\label{fig8}}
\end{figure}

\subsection{Large trial spaces and truncation error}
When the size of the matrices increases, the residuals shown in Figure~\ref{fig8} reach a threshold. After this threshold, truncation error in (finite) 16 digits precision takes over. We show this phenomenon in Figure~\ref{fig9}. Accurate approximation of the enclosures for each individual eigenvalues for $n$ large, but chosen below this threshold, are given in Table~\ref{table9}. 

\begin{figure}[t]
\centering
\includegraphics[width=4.5cm]{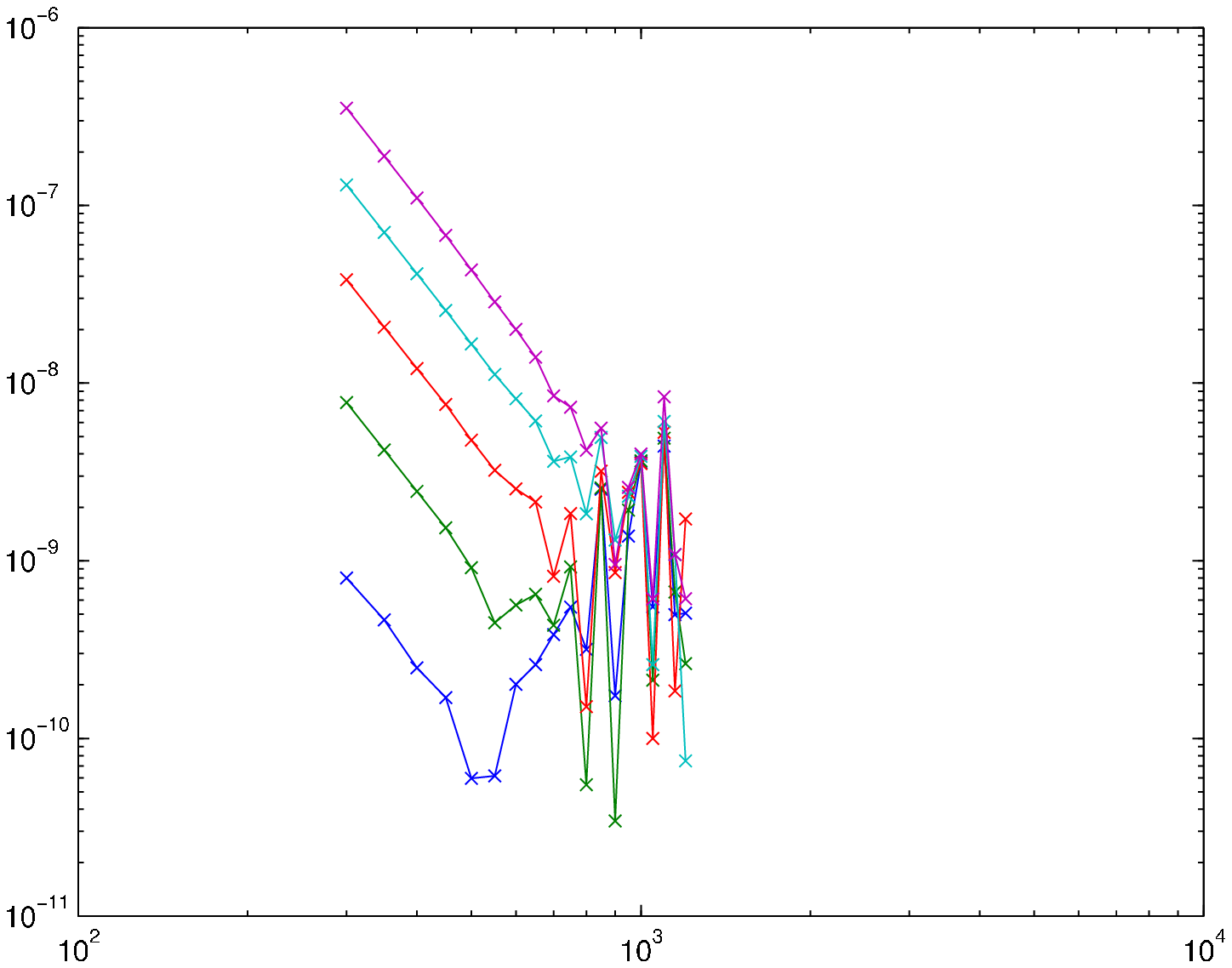}
\includegraphics[width=4.5cm]{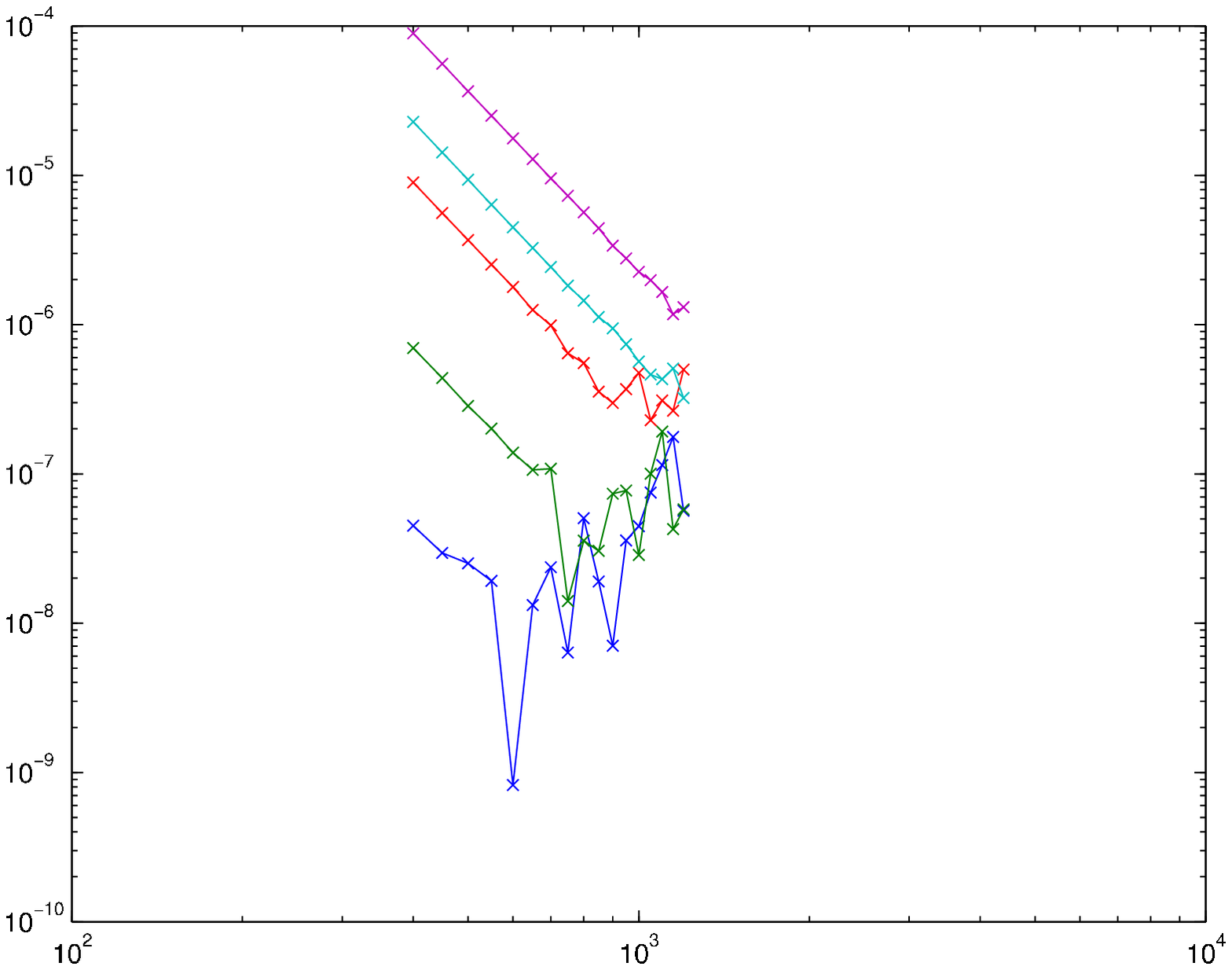}
\includegraphics[width=2cm]{legend7}
\caption{
Loglog plot of the residual $\lambda_{j,\x{up}}-\lambda_{j,\x{low}}$ as $n$ becomes very large for $H_6^{\x{har}}$ (left) and $H_6^{\x{anh}}$ (right). Upper bounds are found by fixing $t=-20$ and lower bounds are found by fixing $t=20$. Here $n$ runs from 300 to 1050. \label{fig9}}
\end{figure}

\begin{table}[t]
\centerline{\begin{tabular}{ |l |c|c||c|c| r| }
  \hline
  $j$ & $n$ & harmonic & $n$ &  anharmonic \\ \hline
  1 &450 &$^{1.0000000027}_{0.9999999943}$ &550 & $1.0603620^{963}_{770}$  \\ \hline
  2 &500 &$^{3.0000000197}_{2.9999999684}$ &650 & $3.79967^{30847}_{29783}$ \\ \hline
  3 &550 &$^{5.0000000452}_{4.9999999516}$ &750 & $7.45569^{81982}_{755622}$ \\ \hline
  4 &650 &$^{7.0000009852}_{6.9999998775}$ &950 & $11.644745^{9324}_{1934}$  \\ \hline
  5 &700 &$^{9.0000001865}_{8.9999997865}$ &1050 & $16.26182^{73730}_{53927}$ \\ \hline
  \end{tabular}}
\caption{Critical value of $n$ before truncation error takes over in the calculation of the upper and lower eigenvalue bounds. Upper bounds are found by fixing $t=-20$ and lower bounds are found by fixing $t=20$. \label{table9}}
\end{table}

\subsection{Influence of the shift on the eigenvalue bounds}
We now examine the influence of the choice of $t$ on the quality of the eigenvalue bounds, as described in Section~\ref{impzimesec}. In Table~\ref{table8} we fixed $t=-20$ in order to compute the upper eigenvalue bounds and $t=20$ in order to compute the lower eigenvalue bounds. Figure~\ref{fig10} shows different choices of $t$ versus $\lambda_{j,\x{up}}-\lambda_{j,\x{low}}$ in semi-log scale. As predicted by Theorem~\ref{lem52}, the further the $t$ moves away from the spectrum, the more accurate the enclosure becomes.

\begin{figure}[t]
\centering
\includegraphics[width=4.5cm]{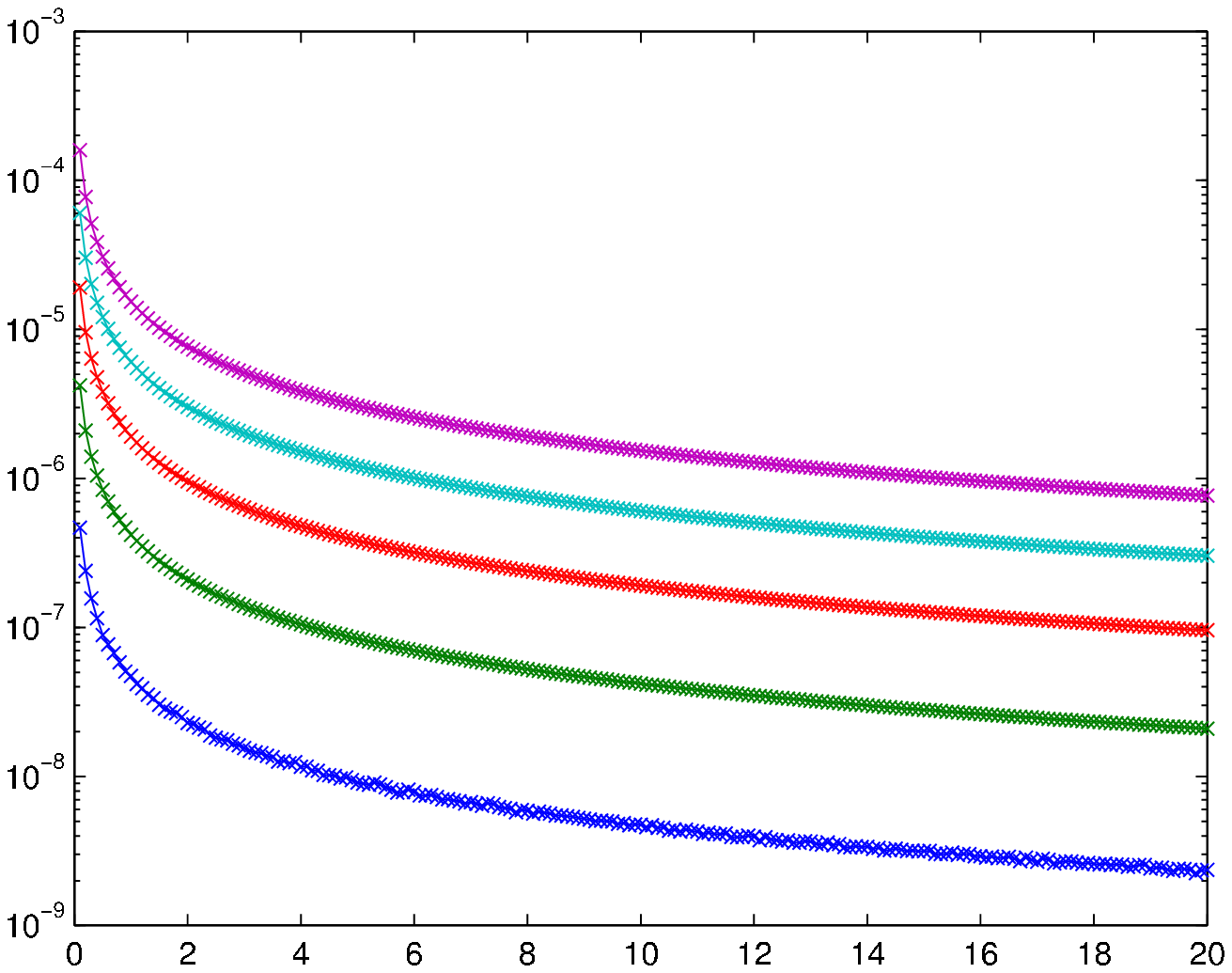}
\includegraphics[width=4.5cm]{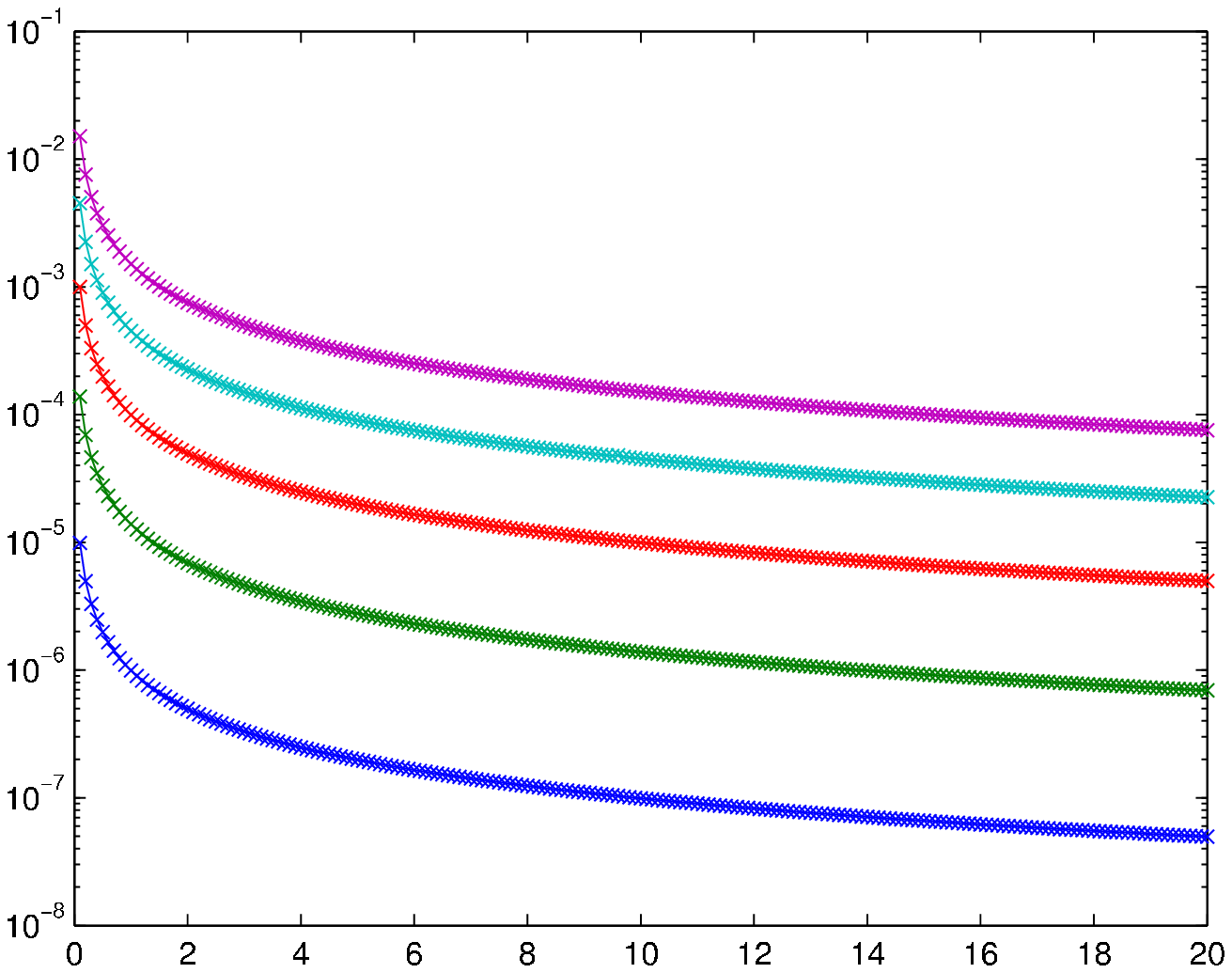}
\includegraphics[width=2cm]{legend7}
\caption{Semi-log plot of the residual $\lambda_{j,\x{up}}-\lambda_{j,\x{low}}$ as the shift $t$ 
moves away from the spectrum for $H_6^{\x{har}}$ (left) and $H_6^{\x{anh}}$ (right). Here $n=200$. \label{fig10}}
\end{figure}

The Galerkin method is widely acknowledged to be the standard approach for computing upper bound for eigenvalues. See Table~\ref{table8}. From the illustration shown in Figure \ref{fig_con}, it is natural to predict that the upper bounds \eqref{upper} approach those provided by the Galerkin method in the regime $t\to -\infty$.

\begin{figure}[t] 
\includegraphics[width=9cm]{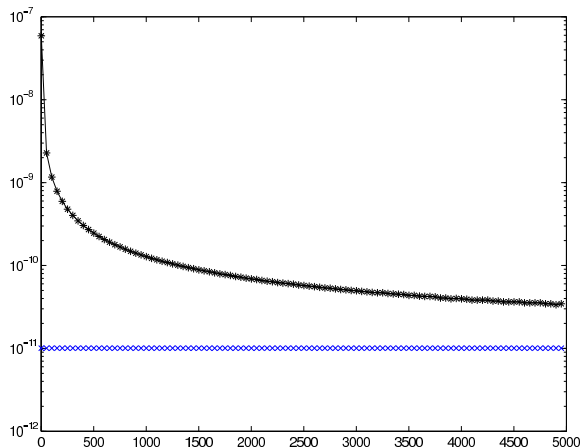}
\caption{Semi-log plot for $H_6^{\x{har}}$, comparing Galerkin upper bounds against those provided by the method of this paper. Here the vertical axis is $\lambda_{1,\x{up}}-1$, the horizontal axis is $-t$, the blue marks correspond to bounds computed by means of the Galerkin method. \label{fig_con} }
\end{figure}

\section*{Acknowledgements}
Financial support was provided by the Engineering and Physical Sciences Research Council (grant number EP/I00761X/1) and King Abdulaziz University.

\bibliographystyle{plain}

 \end{document}